\documentclass[11pt]{amsproc}
\usepackage[T1]{fontenc}
\usepackage{lmodern}
\usepackage[margin=1.12in]{geometry}
\usepackage{amssymb,mathtools,microtype}
\usepackage[dvipsnames]{xcolor}
\usepackage{enumitem}
\usepackage{hyperref}
\hypersetup{colorlinks=true,linkcolor=MidnightBlue,citecolor=MidnightBlue,
  urlcolor=MidnightBlue,pdftitle={Bergman's amalgamation problem over the circle},
  pdfauthor={Andreas Thom},pdfsubject={Hermitian positivity, unitary characters, and compact amalgamation}}
\numberwithin{equation}{section}
\newtheorem{theorem}{Theorem}[section]
\newtheorem*{mainproposition}{Proposition A}
\newtheorem*{maintheorem}{Theorem B}
\newtheorem*{positivityproposition}{Proposition C}
\newtheorem{proposition}[theorem]{Proposition}
\newtheorem{lemma}[theorem]{Lemma}

\theoremstyle{definition}

\theoremstyle{remark}
\newtheorem{remark}[theorem]{Remark}
\newcommand{\Z}{\mathbb Z}

\newcommand{\C}{\mathbb C}

\newcommand{\T}{S^1}
\newcommand{\U}{\mathrm U}
\newcommand{\SU}{\mathrm{SU}}
\newcommand{\SO}{\mathrm{SO}}

\DeclareMathOperator{\diag}{diag}

\DeclareMathOperator{\tr}{tr}

\allowdisplaybreaks[1]
\title[Bergman's amalgamation problem over the circle]{Bergman's amalgamation problem over the circle}
\author{Andreas Thom}
\address{A.~Thom, Institut f\"ur Geometrie, Fakult\"at Mathematik,
TU Dresden, 01062 Dresden, Germany}
\email{andreas.thom@tu-dresden.de}
\date{11 September 2026}

\begin{document}
\begin{abstract}
We prove that $S^1$, $\SU(2)$, and $\SO(3)$ are compact amalgamation bases:
for each of these groups, any two compact groups with specified
embedded copies of it embed in a compact group in which those copies agree.
For compact Lie groups, the amalgam can be chosen to be a finite-dimensional
unitary group.

This research was conducted with substantial AI assistance, including in
the development of mathematical constructions and proof arguments.
\end{abstract}
\maketitle
\tableofcontents

\section{Introduction}\label{sec:intro}

The compact amalgamation problem asks whether two compact groups with a
specified common subgroup can be embedded in a single compact group so
that the two embeddings agree on that subgroup. We answer this question in
the affirmative when the common subgroup is the circle $S^1$, $\SU(2)$ and $\SO(3)$, thereby
settling several cases of G.~M. Bergman's question
\cite[Question~20]{Bergman}.

The crucial ingredient in the proof is the Hermitian Positivstellensatz of D.~W. Catlin and
J.~P. D'Angelo \cite{CD}, recalled in Section~\ref{sec:CD}. Under the
positivity hypotheses stated there, all sufficiently high powers of a
Hermitian algebraic kernel are sums of Hermitian squares. Applying this
theorem to kernels on a Grassmannian proves the required character-positivity
statement. These characters are then used
to construct faithful representations of the ambient groups with matching
restrictions to the circle.

Throughout, compact groups are understood to be Hausdorff, and we write
$\T=\{z\in\C:|z|=1\}$ for the circle group.
Let $H,A,B$ be compact groups and let
$\iota_A:H\hookrightarrow A$ and $\iota_B:H\hookrightarrow B$ be continuous
embeddings. A \emph{compact amalgam} for these embeddings
is a compact Hausdorff group $K$ together with continuous embeddings
$j_A:A\hookrightarrow K$ and $j_B:B\hookrightarrow K$ such that
$ j_A\circ\iota_A=j_B\circ\iota_B.$
Thus the two groups must embed faithfully, but their images may intersect
in more than the common subgroup $H$. A compact group $A$ is a
\emph{compact amalgamation base} if every pair
of embeddings of it into compact groups as above admits a compact amalgam.

\medskip

Our main results are the following.

\begin{mainproposition}\label{prop:lie}
Let $A,B$ be compact Lie groups, let $H\in\{\T,\SU(2),\SO(3)\}$, and let
$\iota_A:H\hookrightarrow A$ and $\iota_B:H\hookrightarrow B$ be continuous
embeddings. Then there exist an integer $d\ge1$ and faithful continuous
unitary representations $\pi:A\longrightarrow\U(d)$ and
$\tau:B\longrightarrow\U(d)$ such that
$\pi\circ\iota_A=\tau\circ\iota_B$. Thus the compact amalgam can be
chosen to be the finite-dimensional unitary group $\U(d)$.
\end{mainproposition}

This proposition can then be used to deduce our main result:

\begin{maintheorem}\label{thm:main}
The groups $S^1$, $\SU(2)$, and $\SO(3)$ are compact amalgamation bases.
\end{maintheorem}

For the circle case of \hyperref[prop:lie]{Proposition A}, let $A$ and $B$
be compact Lie groups with specified circle embeddings.
A finite-dimensional unitary representation of the circle is
determined by its integer weights: its character is a sum of monomials
$z^a$, counted with multiplicity. We therefore seek faithful
representations of $A$ and $B$ with the same restricted weight
multiplicities. The difficulty is to realize a common circle character as
the restriction of a genuine representation of each ambient group, rather
than merely as a formal difference of characters. This strategy of matching
restricted characters by composing unitary representations was already
proposed by G.~M. Bergman in the discussion following
\cite[Question~20, pp.~273--274]{Bergman}.

\medskip

Our construction starts with faithful representations
$A\to\U(m)$ and $B\to\U(\ell)$, with circle weight lists
$a_1,\ldots,a_m$ and $b_1,\ldots,b_\ell$. The restriction theorem for
representations of closed subgroups allows us to arrange that each list
contains the weight $1$; see Lemma~\ref{lem:weight-one}.
We then use the Laurent polynomial
\begin{equation}\label{eq:intro-common}
 \prod_{i=1}^m\prod_{j=1}^{\ell}
 (1+z^{a_i b_j})^n(1+z^{-a_i b_j})^n.
\end{equation}
This expression is unchanged when the two weight lists are interchanged.
For the first ambient group, it is the restriction of the symmetric Laurent
polynomial
\[
 \prod_{i=1}^m f_b(x_i)^n,\qquad
 f_b(z)=\prod_{j=1}^{\ell}(1+z^{b_j})(1+z^{-b_j});
\]
interchanging $a$ and $b$ gives the candidate for the second group. It
remains to show that these candidates are characters of faithful
representations of $\U(m)$ and $\U(\ell)$ for a common exponent $n$.

The relevant character-positivity statement is the following consequence
of the Catlin--D'Angelo theorem, specialized to the products of binomials
that occur in the construction.

Recall that Schur polynomials are the characters of irreducible polynomial
representations of $\U(m)$, expressed in the eigenvalues of a unitary
matrix. Thus nonnegative integer coefficients in the Schur basis give
the multiplicities in a genuine representation.

The main technical ingredient is therefore to establish nonnegative integer coefficients in the Schur basis, which is the content of the following Proposition:

\begin{positivityproposition}\label{prop:schur-positive}
Let
\[
 q(z)=\prod_{j=1}^s(1+z^{k_j}),\qquad k_j\in\Z_{\ge0},
\]
where at least one $k_j$ equals $1$. For every integer $m\ge1$ there
is $n_0=n_0(q,m)$ such that, for every integer $n\ge n_0$, the symmetric
polynomial
\[
 \prod_{i=1}^m q(x_i)^n
\]
has nonnegative integer coefficients in the Schur basis in $m$ variables.
Thus it is the character of a finite-dimensional polynomial
representation of $\U(m)$.
\end{positivityproposition}

To apply \hyperref[prop:schur-positive]{Proposition C} to $f_b$, set
$s_b=\sum_j|b_j|$. The polynomial
\[
 q_b(z)=z^{s_b}f_b(z)=\prod_{j=1}^{\ell}(1+z^{|b_j|})^2
\]
has nonnegative integer coefficients and a positive constant coefficient.
Since the list $b$ contains $1$, the product contains a factor $1+z$. Thus
\hyperref[prop:schur-positive]{Proposition C} makes $\prod_i q_b(x_i)^n$ a polynomial
character for every sufficiently large $n$. Tensoring with $\det^{-ns_b}$
gives the required Laurent character $\prod_i f_b(x_i)^n$.
The factor associated with weight $1$ also ensures faithfulness, as shown
in Proposition~\ref{prop:character}. Interchanging the lists completes
the construction once $n$ is chosen beyond both positivity thresholds.
Composing these representations with the initial embeddings of $A$ and $B$
gives faithful representations whose restricted circle characters are both
\eqref{eq:intro-common}. The restrictions are therefore unitarily equivalent,
and conjugating one representation makes them equal. This proves the
circle case of \hyperref[prop:lie]{Proposition A}.

The proof of \hyperref[prop:schur-positive]{Proposition C} uses the kernel
$\det q(ZW^*)$ on a matrix chart of $\operatorname{Gr}(m,2m)$.
It extends to a Hermitian algebraic kernel on a power of the Pl\"ucker
line bundle. The Catlin--D'Angelo theorem
then makes all sufficiently high powers sums of Hermitian squares.
Writing $K$ for this kernel, this means a decomposition
\[
 K(x,\overline y)^n=\sum_\nu s_\nu(x)\overline{s_\nu(y)}
\]
with holomorphic sections $s_\nu$. The precise setting and hypotheses
are recalled in Section~\ref{sec:CD}; the Grassmannian and its Pl\"ucker
line bundle are described in Section~\ref{sec:binomial-kernels}.
Restriction to graph planes of unitary matrices gives the claimed
character positivity, as explained in Section~\ref{sec:unitary-restriction}.

\medskip

The construction is
qualitative; we do not obtain a useful bound for the dimensions of the
amalgamating representations. However, effective versions of the geometric
positivity theorem are available \cite{TTeff}.

The passage to arbitrary compact Hausdorff groups is a special case of
G.~M. Bergman's reduction to compact Lie groups \cite[Section~1]{Bergman}.
It uses Peter--Weyl: finite-dimensional unitary representations separate
points, and finitely many suffice to detect the circle faithfully.
Applying the Lie-group result to suitable image groups and taking a
product proves the circle case of \hyperref[thm:main]{Theorem B}.

Joseph Israel \cite[Proposition~I.1.2 and \S\S I.1.3--I.1.4, pp.~3--4]{Israel}
proved that amalgamation over $\SU(2)$ and over
$\mathrm{PSU}(2)\cong\SO(3)$ reduces to amalgamation over the circle.
More generally, for a compact connected Lie group, the amalgamation-base
property follows from that of a closed subgroup containing a maximal
torus. Combining this reduction with the circle case completes the proof
of \hyperref[thm:main]{Theorem B}. We give the deduction in
Section~\ref{sec:su2}.

The connection between amalgamation and convolution unimodality already
appears in J.~S. Israel's thesis \cite{Israel}. His public abstract describes
partial amalgamation and completability results based on the eventual
strong unimodality theorem of A.~M. Odlyzko and L.~B. Richmond \cite{OR}
and a higher-difference property for symmetric Laurent polynomials.
Odlyzko--Richmond also observed that their method yields
stronger variation-diminishing properties of high convolutions
\cite[Section~1, discussion following Theorem~2]{OR}.
M.~Joswig, M.~Kummer, C.~H. Yun, and the author formulate the problem as
faithful Schur-positive character matching
\cite[Questions~2.3 and~5.1]{JKTY}. For circle embeddings in $\SU(3)$,
they establish rational symmetric lifts and common Laurent polynomials
\cite[Theorem~4.4 and Corollaries~4.6--4.8]{JKTY}, and ask for Schur
positivity of a specified family \cite[Conjecture~4.9]{JKTY}.
The common polynomial \eqref{eq:intro-common} used here is different
from that family.

\medskip

The paper is organized as follows. Section~\ref{sec:positivity} proves \hyperref[prop:schur-positive]{Proposition C}
from the Catlin--D'Angelo theorem. Section~\ref{sec:characters} applies
it to the Laurent polynomials associated with circle weights and checks
faithfulness. Section~\ref{sec:construction} matches the circle
restrictions and proves the circle case of \hyperref[prop:lie]{Proposition A}.
This then also passes to arbitrary compact Hausdorff groups,
proving the circle case of \hyperref[thm:main]{Theorem B}.
Section~\ref{sec:su2} recalls J.~S. Israel's reduction to deduce the
$\SU(2)$ and $\SO(3)$ cases and discusses open problems.

\medskip

\noindent\textbf{AI declaration.} The results and proofs presented here were generated almost entirely by
OpenAI's Codex. I formulated the research question and guided the
investigation through suggestions, critical discussion, and requests for
explanations and revisions. Codex also searched the literature and drafted
the manuscript. I reviewed its responses and worked through the arguments
in dialogue with the system. As a result, the initial 13-page draft was
considerably improved and made easier to understand.
Conventions for assigning credit to research
produced in this way have yet to be established.

\section{Hermitian positivity and unitary characters}\label{sec:positivity}

\subsection{The Catlin--D'Angelo theorem}\label{sec:CD}

Let $L$ be a holomorphic line bundle over a compact complex manifold $X$.
A \emph{Hermitian algebraic kernel on $L$} is a finite Hermitian linear
combination of products of holomorphic sections:
\[
 R(x,\overline y)=\sum_{\alpha,\beta}
 C_{\alpha\beta}s_\alpha(x)\overline{s_\beta(y)},
 \qquad C_{\alpha\beta}=\overline{C_{\beta\alpha}}.
\]
Here $s_\alpha\in H^0(X,L)$, and evaluations may be read in local
holomorphic trivializations. Positivity on the diagonal means
$R(x,\overline x)>0$ in such a trivialization for every $x$.
This defines a Hermitian metric on the dual bundle $L^*$.
The inequality
\begin{equation}\label{eq:strict-cs}
 |R(x,\overline y)|^2
 < R(x,\overline x)R(y,\overline y)\qquad(x\ne y)
\end{equation}
is independent of the trivializations. The same is true of the
$(1,1)$-form $\mathrm i\partial\bar\partial\log R(x,\overline x)$;
its strict positivity means that the induced metric on $L^*$ has
strictly negative curvature.

\begin{theorem}[Catlin--D'Angelo]\label{thm:CD}
Let $R$ be a Hermitian algebraic kernel on a holomorphic line bundle $L$
over a compact complex manifold $X$. Suppose that $R$ is positive on
the diagonal, satisfies \eqref{eq:strict-cs}, and
\[
 \mathrm i\partial\bar\partial\log R(x,\overline x)>0.
\]
Then for every sufficiently large integer $n$ there are finitely many
sections $s_\nu\in H^0(X,L^n)$ such that
\begin{equation}\label{eq:CD-squares}
 R(x,\overline y)^n=\sum_\nu s_\nu(x)\overline{s_\nu(y)}.
\end{equation}
\end{theorem}

This follows from \cite[Theorem~1, p.~44]{CD}, with the auxiliary bundle
$E$ trivial and with their negative line bundle equal to $L^*$.
Equivalently, it is the special case $E$ trivial and $P=1$ of
\cite[Theorem~2, p.~7]{TT}.
The two strict hypotheses are the strong global Cauchy--Schwarz
condition of \cite[p.~5]{TT}. The cited theorem gives a maximal sum of
squares; the weaker conclusion \eqref{eq:CD-squares} is sufficient here.
In particular, $R^n$ is a positive semidefinite kernel in any choices
of fibers or trivializations. Quillen's earlier theorem \cite{Quillen}
makes a bihomogeneous Hermitian polynomial that is strictly positive away
from the origin a sum of Hermitian squares after multiplication by a
sufficiently high power of the standard squared norm. The geometric
theorem above allows us to take powers of the particular kernel $R$.

\subsection{Binomial kernels on the Grassmannian}\label{sec:binomial-kernels}

For a complex matrix $M$ and $p\ge1$, its Schatten $p$-norm is
\[
 \|M\|_p=\bigl(\tr((MM^*)^{p/2})\bigr)^{1/p}.
\]
The instance of H\"older's inequality for Schatten norms that we need
\cite[Corollary~IV.2.6]{Bhatia} says that, for square matrices
$M_1,\ldots,M_{2k}$ of the same size,
\[
 |\tr(M_1\cdots M_{2k})|\le\prod_{j=1}^{2k}\|M_j\|_{2k}.
\]

For an $m$ by $m$ matrix $Z$, put
\[
 \mathcal E(Z)=\bigoplus_{r=0}^m\bigwedge^r Z.
\]
This construction respects products and adjoints, and
$\tr\mathcal E(Z)=\det(I+Z)$. Applying the inequality to $2k$ factors
alternating between $\mathcal E(Z)$ and $\mathcal E(W)^*$ gives
\begin{align}
 |\det(I+(ZW^*)^k)|
 &=\bigl|\tr\bigl[(\mathcal E(Z)\mathcal E(W)^*)^k\bigr]\bigr|\notag\\
 &\le\|\mathcal E(Z)\|_{2k}^{k}\|\mathcal E(W)\|_{2k}^{k}\notag\\
 &=\sqrt{\det(I+(ZZ^*)^k)\det(I+(WW^*)^k)}.
 \label{eq:matrix-cs}
\end{align}
Let $X=\operatorname{Gr}(m,2m)$ be the Grassmannian of $m$-dimensional
complex subspaces of $\C^{2m}$. The Pl\"ucker embedding sends a subspace
with basis $v_1,\ldots,v_m$ to the projective point
$[v_1\wedge\cdots\wedge v_m]$ in
$\mathbb P(\bigwedge^m\C^{2m})$. Its Pl\"ucker line bundle
$\mathcal O_X(1)$ is the pullback of the hyperplane bundle under this
embedding. Represent planes by full-row-rank frames
$V=(A\ B)$ and $V'=(C\ D)$, with $m$ by $m$ blocks. For $k\ge1$, set
\begin{equation}\label{eq:grassmann-kernel}
 K_k(V,\overline{V'})
 =\prod_{\alpha^k=(-1)^{k+1}}\det(AC^*+\alpha BD^*).
\end{equation}
The identity
$\prod_{\alpha^k=(-1)^{k+1}}(1+\alpha t)=1+t^k$ gives, on the graph
chart $V=(I\ Z)$, $V'=(I\ W)$,
\begin{equation}\label{eq:graph-kernel}
 K_k(Z,\overline W)=\det(I+(ZW^*)^k).
\end{equation}
The Pl\"ucker coordinates of a frame are its maximal minors.
By Cauchy--Binet, each determinant in \eqref{eq:grassmann-kernel} is a
sesquilinear form in these coordinates. Frame changes multiply
their product by $(\det E)^k\overline{(\det F)^k}$, and the set of
$\alpha$ is closed under conjugation. Thus $K_k$ is a Hermitian
algebraic kernel on $\mathcal O_X(k)$.

These kernels have positive diagonal globally. Indeed, normalize $V$
so that $P=AA^*$ and $Q=BB^*$ satisfy $P+Q=I$. Then $P,Q$ commute and
\begin{equation}\label{eq:global-diagonal}
 K_k(V,\overline V)=\det(P^k+Q^k)>0,
\end{equation}
since each eigenvalue of $P^k+Q^k$ has the form $t^k+(1-t)^k>0$
with $0\le t\le1$. Inequality~\eqref{eq:matrix-cs} proves weak global
Cauchy--Schwarz on the dense graph chart, and continuity extends it
to every pair of planes.

Weak Cauchy--Schwarz also gives semipositive curvature; see
J.~P. D'Angelo and D.~Varolin \cite[Section~5, pp.~228--229]{DV}.
For completeness, in a local
holomorphic trivialization, fix $x_0$ and observe that
\[
 \log K_k(z,\overline z)-\log|K_k(z,\overline{x_0})|^2
       +\log K_k(x_0,\overline{x_0})
\]
has a minimum of zero at $x_0$. The polarized factor is nonzero near
$x_0$, so the middle term is pluriharmonic. Consequently
$\mathrm i\partial\bar\partial\log K_k(z,\overline z)\ge0$.

\subsection{Strict positivity and restriction to the unitary group}\label{sec:unitary-restriction}

A function $\chi$ on $\U(m)$ is of \emph{positive type} if, for every
finite list $U_1,\ldots,U_r\in\U(m)$, the matrix
$[\chi(U_iU_j^*)]_{i,j=1}^r$ is positive semidefinite.

\begin{proof}[Proof of Proposition C]
The kernel $K_1(V,\overline{V'})=\det(V(V')^*)$ is the standard
Pl\"ucker kernel. Its Cauchy--Schwarz inequality is strict for distinct
planes, since their Pl\"ucker vectors are not proportional, and its
curvature form is the positive Fubini--Study form restricted along
the Pl\"ucker embedding.

Write $q=\prod_j(1+z^{k_j})$, first omitting factors with $k_j=0$,
and put $K=\prod_jK_{k_j}$. This is a Hermitian algebraic kernel on
$\mathcal O_X(\sum_j k_j)$, positive on the diagonal. All its factors
satisfy weak Cauchy--Schwarz and have semipositive curvature, and at
least one factor is $K_1$. Thus $K$ satisfies both strict hypotheses
of Theorem~\ref{thm:CD}. For every sufficiently large $n$, the kernel
$K^n$ is a sum of Hermitian squares.

Restrict this kernel to graph planes of unitary matrices. In the graph
trivialization it has the form
\[
 K(U,\overline V)^n=\det q(UV^*)^n\qquad(U,V\in\U(m)).
\]
The sum-of-squares expression \eqref{eq:CD-squares} makes every matrix
$[K(U_i,\overline{U_j})^n]_{i,j}$ a sum of positive semidefinite
matrices of rank at most one. Consequently $\chi_n(U)=\det q(U)^n$ is a central
function of positive type on $\U(m)$. Its restriction to the maximal torus is
the symmetric polynomial $\prod_i q(x_i)^n$. The Schur polynomials
$s_\lambda$, indexed by partitions with at most $m$ parts, form a
basis of symmetric polynomials and are the irreducible polynomial
characters of $\U(m)$; see \cite{FH}. Hence
\[
 \chi_n=\sum_\lambda c_\lambda\chi_\lambda
\]
is a finite irreducible-character expansion. Fourier positivity for a
central function of positive type gives $c_\lambda\ge0$
\cite[Sections~3.3 and~5.3]{Folland}: its Fourier
transform in the irreducible representation with character
$\chi_\lambda$ is $(c_\lambda/\dim\chi_\lambda)I$ and is positive
semidefinite.

The symmetric polynomial $\prod_i q(x_i)^n$ has integer coefficients.
Since Schur polynomials form an integral basis of symmetric polynomials
over $\Z$ \cite[Chapter~I, \S3, (3.2), p.~41]{Macdonald}, every
$c_\lambda$ is an integer. Reinserting factors
with $k_j=0$ multiplies the character by a positive integer. This
proves the proposition.
\end{proof}

\section{Faithful characters from circle weights}\label{sec:characters}

We apply \hyperref[prop:schur-positive]{Proposition C} to the polynomials used in
the matching construction. The occurrence of weight $1$ supplies
the factor required by that proposition and also rules out a nontrivial
kernel. Tensoring a polynomial representation of $\U(m)$ with any
integral power of the determinant gives a genuine representation,
even when its character has negative exponents.

\begin{proposition}\label{prop:character}
Let $c=(c_1,\ldots,c_\ell)\in\Z^\ell$ contain the integer $1$, and let
$m\ge1$. Define
\begin{equation}\label{eq:fc}
 f_c(z)=\prod_{j=1}^{\ell}(1+z^{c_j})(1+z^{-c_j}),\qquad
 s_c=\sum_{j=1}^{\ell}|c_j|.
\end{equation}
For all sufficiently large integers $n$,
\begin{equation}\label{eq:Hcm}
 H_{c,m,n}(x_1,\ldots,x_m)=\prod_{i=1}^m f_c(x_i)^n
\end{equation}
is the character of a faithful representation
$\Phi_{c,m,n}:\U(m)\to\U(4^{nm\ell})$.
\end{proposition}

\begin{proof}
The ordinary polynomial
\begin{equation}\label{eq:qc}
 q_c(z):=z^{s_c}f_c(z)
         =\prod_{j=1}^{\ell}(1+z^{|c_j|})^2
\end{equation}
has nonnegative integer coefficients and a positive constant coefficient.
Since some $c_j=1$, this binomial product contains $1+z$ and
satisfies \hyperref[prop:schur-positive]{Proposition C}. Factors with $c_j=0$ are
simply the constant $4$.

By \hyperref[prop:schur-positive]{Proposition C}, for every sufficiently large $n$
there is a polynomial representation of $\U(m)$ with character
$\prod_i q_c(x_i)^n$. Tensoring with $\det^{-ns_c}$ gives
\eqref{eq:Hcm}. At the identity its character equals
$f_c(1)^{nm}=4^{nm\ell}$.

To check faithfulness, diagonalize an arbitrary $u\in\U(m)$, with
eigenvalues $x_1,\ldots,x_m\in\T$. We have
\begin{equation}\label{eq:faithful-bound}
 \chi_{\Phi_{c,m,n}}(u)
 =\prod_{i=1}^m\prod_{j=1}^{\ell}|1+x_i^{c_j}|^{2n}
 \le4^{nm\ell}.
\end{equation}
If $u$ lies in the kernel, its character equals its dimension, so equality
holds. Every factor in the product is nonnegative and at most its value
at $1$, so equality forces $|1+x_i^{c_j}|=2$ for every $i,j$. Taking a
$j$ with $c_j=1$ gives $x_i=1$ for every $i$. A unitary matrix with all
eigenvalues $1$ is the identity. Hence the representation is faithful.
\end{proof}

\begin{remark}
Faithfulness is proved on the entire group $\U(m)$. It is not inferred
merely from faithfulness of a restriction to one circle. This distinction
avoids any need to analyze possible central kernels after the representation
is pulled back to a compact Lie group.
\end{remark}

\section{Matching characters and compact amalgams}\label{sec:construction}

It remains to choose the initial representations and apply the preceding
construction twice, once for each weight list. The restriction theorem for
representations of closed subgroups supplies the weight-$1$ hypothesis of
Proposition~\ref{prop:character}.

\begin{lemma}\label{lem:weight-one}
Let $G$ be a compact Lie group and $\iota:\T\hookrightarrow G$ a continuous
embedding. There is a faithful finite-dimensional unitary representation
$R$ of $G$ whose restriction along $\iota$ contains circle weight $1$.
\end{lemma}

\begin{proof}
Choose a faithful finite-dimensional unitary representation $R_0$ of $G$
\cite[Theorem~5.13]{Folland}.
By \cite[Chapter~III, Theorem~(4.5), p.~137]{BtD}, the circle character
$z\mapsto z$ occurs in the restriction of an irreducible representation
$S$ of $G$; hence $R_0\oplus S$ is faithful and contains weight $1$.
\end{proof}

\begin{proof}[Proof of Proposition A for $H=\T$]
By Lemma~\ref{lem:weight-one}, we may choose faithful representations
\[
 R_A:A\to\U(m),\qquad R_B:B\to\U(\ell)
\]
whose respective restricted weight lists
$a=(a_1,\ldots,a_m)$ and $b=(b_1,\ldots,b_\ell)$ both contain $1$.
These representations and dimensions are now fixed.

Apply Proposition~\ref{prop:character} to the list $b$ with ambient rank
$m$, and to the list $a$ with ambient rank $\ell$. Choose a common integer
$n$ beyond both thresholds. The resulting faithful representations
\[
 \Phi_{b,m,n}:\U(m)\to\U(d),\qquad
 \Phi_{a,\ell,n}:\U(\ell)\to\U(d),\qquad
 d=4^{nm\ell},
\]
have the same dimension. Put
\[
 \pi=\Phi_{b,m,n}\circ R_A,\qquad
 \tau=\Phi_{a,\ell,n}\circ R_B.
\]
They are faithful as compositions of faithful representations. Their
circle characters are
\begin{align*}
 \chi_{\pi\circ\iota_A}(z)
 &=\prod_{i=1}^m f_b(z^{a_i})^n
 =\prod_{i=1}^m\prod_{j=1}^{\ell}
       (1+z^{a_i b_j})^n(1+z^{-a_i b_j})^n\\
 &=\prod_{j=1}^{\ell}f_a(z^{b_j})^n
 =\chi_{\tau\circ\iota_B}(z).
\end{align*}
Finite-dimensional unitary representations of $\T$ split as direct sums
of the characters $z\mapsto z^q$, $q\in\Z$. Equality of their characters
therefore means equality of all weight multiplicities. Choosing orthonormal
bases in corresponding weight spaces gives a single unitary conjugation
that identifies the two restricted homomorphisms. After that conjugation,
$\pi$ and $\tau$ are the embeddings required for the amalgam. A continuous
injective map from a compact space to a Hausdorff space is a topological
embedding, so these are embeddings in the category of compact groups.
\end{proof}

G.~M. Bergman proved that a compact Lie group is a compact amalgamation
base if and only if it is an amalgamation base among compact Lie groups
\cite[Section~1, statements~(2) and~(3); see also pp.~276--278]{Bergman}.
The reduction is also recalled in \cite[Section~2]{JKTY}.
Applying this equivalence to the circle case of
\hyperref[prop:lie]{Proposition A} proves
\hyperref[thm:main]{Theorem B} for $S^1$.

\section{Amalgamation over \texorpdfstring{$\SU(2)$ and $\SO(3)$ and further questions}{SU(2) and SO(3)}}\label{sec:su2}

J.~S. Israel explicitly gives the reductions used here
\cite[Proposition~I.1.2 and \S\S I.1.3--I.1.4, pp.~3--4]{Israel}.
The underlying fact that two representations of a compact connected
Lie group are equivalent when their restrictions to a maximal torus are
equivalent also appears in G.~M. Bergman's discussion of relative inner
amalgamation \cite[Section~11, p.~279]{Bergman}. We repeat the proof for convenience:

\begin{proof}[Proof of Theorem B for $\SU(2)$ and $\SO(3)$]
Let $A,B$ be compact groups with specified embeddings of
$H\in\{\SU(2),\SO(3)\}$, and let $T$ be a maximal torus of $H$.
In both cases $T\cong\T$. Restrict the two specified embeddings
of $H$ into $A,B$ to $T$. The circle case proved in
Section~\ref{sec:construction} gives a compact amalgam $K$, with embeddings
$j_A,j_B$. Choose a family
$(\rho_\alpha:K\to\U(d_\alpha))_\alpha$ separating points of $K$
by Peter--Weyl.

For each $\alpha$, the two $H$-representations obtained by composing
with $\rho_\alpha j_A$ and $\rho_\alpha j_B$ agree on $T$.
Every element of $H$ is conjugate into $T$, so their characters agree
on $H$. They are therefore unitarily equivalent. Choose a unitary intertwiner
$u_\alpha$. The maps
\[
 a\longmapsto(\rho_\alpha(j_A(a)))_\alpha,\qquad
 b\longmapsto(u_\alpha\rho_\alpha(j_B(b))u_\alpha^{-1})_\alpha
\]
agree on $H$ and are injective, since coordinatewise conjugation does
not change kernels. They are continuous, and compactness makes them
embeddings into the compact group $\prod_\alpha\U(d_\alpha)$.

If $A,B$ are compact Lie groups, the circle case of
\hyperref[prop:lie]{Proposition A} provides $K=\U(D)$. Using its defining representation alone in the same
argument gives a single unitary conjugation and a unitary amalgam.
This also proves \hyperref[prop:lie]{Proposition A} for
$H\in\{\SU(2),\SO(3)\}$.
\end{proof}

\hyperref[thm:main]{Theorem B} settles the rank-one special unitary
case. The higher-rank part of G.~M. Bergman's question remains open
\cite[Question~20]{Bergman}; see also \cite[Section~1]{JKTY}:
\begin{quote}
For each $n\ge3$, is $\SU(n)$ a compact amalgamation base?
\end{quote}
The maximal torus of $\SU(n)$ has dimension $n-1$, and equality of
characters on a single circle no longer determines a representation.
For example, the standard representation of $\SU(n)$ and its dual are
inequivalent when $n\ge3$, but both have character
\[
 z+z^{-1}+n-2
\]
on the circle $\diag(z,z^{-1},1,\ldots,1)$. Thus the deduction for
$\SU(2)$ does not extend by simply choosing a circle in a higher-rank
group. The question is to match characters on the entire common subgroup;
the present argument neither establishes this nor gives a counterexample
for $\SU(n)$ with $n\ge3$.

\section*{Acknowledgments}

I thank George Bergman for many helpful comments and remarks that  improved the exposition of the results.

\end{document}